\documentclass[12pt]{amsart}
\usepackage{latexsym, amsthm, amscd, euscript, float, subfig, amssymb, amsmath, graphicx,xcolor}
\usepackage{orcidlink}
\usepackage{hyperref}

\newtheorem{theorem}{Theorem}[section]
\newtheorem{lemma}[theorem]{Lemma}

\newtheorem{ca}{Case}

\numberwithin{equation}{section}

\newcounter{minutes}
\divide\time by 60
\newcounter{hours}
\multiply\time by 60
\addtocounter{minutes}{-\time}

\newcommand{\ds}{\displaystyle}

\newcommand{\R}{\mathbb{R}}

\newcommand{\Rn}{ {\mathbb{R}^n} }

\renewcommand{\tanh}{\,\textnormal{tanh}}

\def\be{\begin{equation}}
\def\ee{\end{equation}}

\newcommand{\bca}{\begin{ca}}
\newcommand{\eca}{\end{ca}}
\newcommand{\bsca}{\begin{sca}}
\newcommand{\esca}{\end{sca}}

\newcommand{\bcl}{\begin{cl}}
\newcommand{\ecl}{\end{cl}}

\newcommand{\bscl}{\begin{scl}}
\newcommand{\escl}{\end{scl}}

\newcommand{\bcons}{\begin{conjs}}
\newcommand{\econs}{\end{conjs}}
\newcommand{\bprop}{\begin{propo}}
\newcommand{\eprop}{\end{propo}}
\newcommand{\br}{\begin{rem}}
\newcommand{\er}{\end{rem}}
\newcommand{\brs}{\begin{rems}}
\newcommand{\ers}{\end{rems}}
\newcommand{\bo}{\begin{obser}}
\newcommand{\eo}{\end{obser}}
\newcommand{\bos}{\begin{obsers}}
\newcommand{\eos}{\end{obsers}}
\newcommand{\bpf}{\begin{pf}}
\newcommand{\epf}{\end{pf}}
\newcommand{\ba}{\begin{array}}
\newcommand{\ea}{\end{array}}
\newcommand{\beq}{\begin{eqnarray}}
\newcommand{\beqq}{\begin{eqnarray*}}
\newcommand{\eeq}{\end{eqnarray}}
\newcommand{\eeqq}{\end{eqnarray*}}

\begin{document}
\vspace*{-2cm}
\title[On the average scale-invariant Cassinian metric]
{On the average scale-invariant Cassinian metric}

\author[M. Mohapatra]{Manas Mohapatra}
\address{ Department of Mathematics, Kalinga Institute of Industrial Technology (KIIT),
Patia, Bhubaneswar, Odisha, India 751024
}
\email{manas.mohapatrafma@kiit.ac.in}

\author[A. Rasila]{Antti Rasila$^{*}$  \orcidlink{0000-0003-3797-942X}}
\address{Department of Mathematics and Statistics,
Department of Mathematics with Computer Science, Guangdong Technion, 241 Daxue Road, Jinping District, Shantou, Guangdong 515063, People's Republic of China and Department of Mathematics, Technion - Israel Institute of Technology, Haifa 32000, Israel\newline
\href{https://orcid.org/0000-0003-3797-942X}{{\tt https://orcid.org/0000-0003-3797-942X}}
}
\email{antti.rasila@iki.fi; antti.rasila@gtiit.edu.cn}

\author[M. Vuorinen]{Matti Vuorinen \orcidlink{0000-0002-1734-8228}}
\address{Department of Mathematics and Statistics,
University of Turku, Finland\newline
\href{https://orcid.org/0000-0002-1734-8228}{{\tt https://orcid.org/0000-0002-1734-8228}}
}
\email{vuorinen@utu.fi}

\date{\today}

\def\thefootnote{}
\footnotetext{ \texttt{\tiny File:~\jobname .tex,
          printed: \number\day-\number\month-\number\year,
          \thehours.\ifnum\theminutes<10{0}\fi\theminutes}
} \makeatletter\def\thefootnote{\@arabic\c@footnote}\makeatother

%
%
%
%

\begin{abstract}
We establish geometric relationships between the average scale-invariant Cassinian metric and other hyperbolic type metrics.  We study the local convexity properties of the scale-invariant metric balls in Euclidean once punctured spaces. In addition, we study Lipschitz-conditions with respect to these metrics.


\end{abstract}

\maketitle
\thispagestyle{empty}
\section{Introduction}
Geometric function theory studies local and global behavior of functions $f:X \to Y$ where $X$ and $Y$ are metric spaces, usually subdomains of the Euclidean
space $\mathbb{R}^n, n\ge 2.$ Various metrics, distance measuring functions between pairs of points $x,y \in X$ are here key tools. In particular, metrics of hyperbolic
type measure not only the distance between points but also take into account the distances of the points $x$ and $y$ to the boundary $d(x, \partial X)$ and $d(y, \partial X)$ like
the hyperbolic metric of the unit disk does.
 Metrics of hyperbolic type have been studied
by many authors during the past few decades and most recently by \cite{afv,fkv,frv,fv,rv,sop}. For more background information about this topic, see e.g. \cite{HKV}.


Recently, Aksoy, Ibragimov, and Whiting studied the averages of one-point hyperbolic-type metrics in \cite{AIW} and established a connection to the Gromov hyperbolicity of the induced metric spaces. In general, a metric produced through
addition of two Gromov hyperbolic metrics is not Gromov hyperbolic \cite[Lemma~4.4]{AIW}. 
However, averages of scale-invariant Cassinian metrics are Gromov hyperbolic \cite[Theorem~4.2]{AIW}. Some related results have been recently obtained by Mocanu in \cite{Moc}. A scale invariant Cassinian metric for $D\subsetneq \Rn$ was introduced by Ibragimov in \cite{Ibr16}, and it has been recently studied in \cite{WXV}. It is defined for $x,y\in D$ by 
\begin{equation}
\label{casmet}
\tilde{\tau}_D(x,y):=\log\left(1+\sup_{p\in \partial D}\frac{|x-y|}{\sqrt{|x-p||p-y|}}\right).
\end{equation}

Recall that a metric space $(X,d)$ is Ptolemaic if for any choice of four points $x_1,x_2,x_3,x_4\in X$, we have 
$$
d(x_1,x_2)d(x_3,x_4)\le d(x_1,x_3)d(x_2,x_4)+d(x_1,x_4)d(x_2,x_3).
$$
The formula \eqref{casmet} does not define a metric in an arbitrary metric space (see, \cite[p.~4]{AIW}). However, if $(D,|\cdot|)$ is a Ptolemaic space, then $\tilde{\tau}_D$ is a metric. 
For a general metric space $(X,d)$, the scale-invariant Cassinian metric is defined on the punctured space $X\setminus \{p\}, \,\,p\in X$, by

$$\tau_p(x,y):=\log\left(1+2\frac{d(x,y)}{\sqrt{d(x,p)d(y,p)}}\right), \quad x,y\in X\setminus \{p\}.
$$

For $k\ge 1$, let $p_1,p_2,\ldots,p_k\in X$ and $D=X\setminus \{p_1,p_2,\ldots,p_k\}$. The average scale-invariant Cassinian metric for $D$ is defined by

$$
\hat{\tau}_D(x,y):=\frac{1}{k} \sum_{i=1}^k \tau_{p_i}(x,y)=\frac{1}{k} \sum_{i=1}^k \log \left(1+2 \frac{d(x,y)}{\sqrt{d(x,p_i)d(y,p_i)}} \right).
$$
Ibragimov in \cite{Ibr11} introduced a metric, $u_X$, which hyperbolizes 
(in the sense of Gromov) the locally compact non-complete metric space $(X,d)$ without changing its quasiconformal geometry, by 
$$u_X(x,y)=2\log \frac{d(x,y)+\max\{{\rm dist}(x, \partial X),{\rm dist}(y, \partial X)\}}{\sqrt{{\rm dist}(x, \partial X)\,
{\rm dist}(y, \partial X)}}, \quad x,y\in X.
$$ 

Various properties of metrics have been studied in several
recent papers, see \cite{CZ,Moc, R,SW,WXV,  ZZPG}.
%


In this paper, we investigate properties of the average scale-invariant Cassinian metric $\hat{\tau}_D$. First we obtain sharp bounds for the metric  $\hat{\tau}_D$ with respect to Ibragimov's metric $u_X$ in cases of once and multiply punctured spaces. Then we show similar inequalities for distance ratio type metrics, also on simply and multiply punctured spaces. We also study inclusion inequalities for the balls of this metric when compared to the ones of the Euclidean metric. Furthermore, we obtain a result concerning the coefficient of quasiconformality for bi-Lipschitz mappings of this metric. Finally, we consider the radii of convexity for balls of this metric, which is also a classical question in this topic (see e.g. \cite{RT} for discussion in the case of the quasihyperbolic metric).

\section{Preliminaries}\label{prel}

Denote by $(X,d)$ a metric space where $d$ is a metric, and let $D\subset X$ be a domain (an open and connected set). 
We denote by $\partial D$ the boundary of the domain $D$. Given $x\in X$ and $r>0$, the open ball centered at $x$ and of radius $r$ is denoted by $B_d(x,r):=\{y\in X \colon\, d(x,y)<r\}$.  For a given $x\in X$, we set $d(x):={\rm dist}(x,\partial D)$.
For real numbers $r$ and $s$, we set $r\vee s=\max\{r,\, s\}$ and $r\wedge s :=\min\{r,\, s\}$.

The {\it distance ratio metric}, $\tilde{j}_D$, is defined by
$$
\tilde{j}_D(x,y):=\log\left(1+\frac{d(x,y)}{d(x)\wedge d(y)}\right), \quad x,y\in D.
$$
The above form of the metric $\tilde{j}_D$, which was first considered in \cite{Vuo85},
is a slight modification of the distance ratio metric, $j_D$, introduced by Gehring and Osgood in \cite{GO79}, which is defined by 
$$
j_D(x,y):=\frac{1}{2}\log \left[\left(1+\frac{d(x,y)}{d(x)}\right) \left(1+\frac{d(x,y)}{d(y)}\right)\right],\quad x,y\in D.
$$
Note that these two metrics were originally defined and studied in proper subsets of $\Rn$. Here we consider them in arbitrary metric spaces.
The $\tilde{j}_D$-metric has been widely studied in the literature; see, for instance, \cite{HKV, Vuo88}. 
These two distance ratio metrics are related by: 
$$
\frac{\tilde{j}_D(x,y)}{2} \le j_D(x,y)\le \tilde{j}_D(x,y).
$$

The triangular ratio metric for $x,y\in D$ is defined by
$$ s_D(x,y):=\sup_{p\in \partial D} \frac{d(x,y)}{d(x,p)+d(y,p)}\in [0,1].
$$
Note that the triangular ratio metric was studied in proper subdomains of the complex plane and in Euclidean $n$-space in \cite{CHKV15}. See also \cite{HKV}. 

We now write the Bernoulli inequality:
\begin{equation}\label{bernoulli}
\log(1+ax) \left\{ \begin{array}{ll}
\le a\log(1+x), & a\ge 1, \quad x>0,\\
\ge a\log(1+x), & a\in (0,1],x>0.
\end{array}
\right.
\end{equation} 

For any two real numbers $a>0$ and $b>0$, it is true that
\begin{equation}\label{am-gm}
\frac{a+b}{2}\ge \sqrt{ab},
\end{equation}
which is the well-known AM-GM inequality.

%
%

\section{Comparison results associated with the scale-invariant Cassinian metric}\label{u-tau}
This section is devoted to the comparison of the average scale-invariant Cassinian metric with other hyperbolic-type metrics. First, we compare the scale-invariant Cassinian metric with the $u_D$-metric in once punctured metric spaces followed by the comparison of the average scale-invariant Cassinian metric with the $u_D$ metric in arbitrary metric spaces.
\begin{theorem} \label{main1}
Let $(X,d)$ be a metric space, and let $D=X\setminus \{p\}$, $p\in X$. Then for $x,y\in D$,
$$\tau_p(x,y)\le u_p(x,y)\le 2\tau_p(x,y),
$$
where $u_p(x,y):= u_{X\setminus\{p\}}(x,y)$. The right inequality is sharp $D= \Rn\setminus\{0\}$. 
\end{theorem}
\begin{proof}
Let $x,y\in D$ with $d(x,p)\ge d(y,p)$.
By definition,
\begin{eqnarray*}
u_p(x,y) &=& 2\log \left(\frac{d(x,y)+d(x,p)}{\sqrt{d(x,p)d(y,p)}}\right)\ge 
2\log \left(1+\frac{d(x,y)}{\sqrt{d(x,p)d(y,p)}}\right)\\
&\ge & \log \left(1+2\frac{d(x,y)}{\sqrt{d(x,p)d(y,p)}}\right)=\tau_p(x,y),
\end{eqnarray*}
where the first inequality follows from the assumption and the second inequality follows
from \eqref{bernoulli}.

To prove the second inequality, it is enough to show that
$$ \frac{d(x,y)+d(x,p)}{\sqrt{d(x,p)d(y,p)}} \le 1+2 \frac{d(x,y)}{\sqrt{d(x,p)d(y,p)}},
$$
i.e, 
$$ d(x,y)+d(x,p)\le 2 d(x,y)+\sqrt{d(x,p)d(y,p)},
$$
which is true because
$$ d(x,y)+d(x,p)\le 2d(x,y)+d(y,p) \le 2d(x,y)+\sqrt{d(x,p)d(y,p)}.
$$
Here the first inequality follows from the triangle inequality and the second one from the assumption. 

To prove the sharpness, consider the punctured Euclidean space $\Rn\setminus\{0\}$ and the points $x=e_1=(1,0,\ldots,0)$, and $y=te_1,\, 0<t<1$.
Now,
$$ \tilde{\tau}_0(x,y)=\log \left(1+\frac{2(1-t)}{\sqrt{t}}\right), \quad 
u_0(x,y)=2\log \left(\frac{2-t}{t}\right),
$$
and
$$ \lim_{t\to 0} \frac{\tilde{\tau}_0}{u_0}= \lim_{t\to 0} \frac{(1+t)(2-t)}{(2+t)(\sqrt{t}+2(1-t))}=\frac{1}{2}.
$$
This completes the proof.
\end{proof}

\begin{theorem} \label{main2}
Let $(X,d)$ be a metric space, and let $p_1,p_2,\ldots,p_k\in X$. Then, for $D=X\setminus \{p_1,p_2,\ldots,p_k\}$ and $x,y\in D$, we have
$$\hat{\tau}_D(x,y)\le u_D(x,y).
$$
The inequality is sharp for  $D=\Rn\setminus\{-e_1,e_1\}$.
\end{theorem}

\begin{proof}
Suppose that $d(x)=d(x,p_r)$ and $d(y)=d(y,p_s)$ and assume that $d(x)\ge d(y)$. Clearly, for all $i=1,2,\ldots,k$, 
$$d(x,p_i)\ge d(x)\, , d(y,p_i)\ge d(y) \mbox{ and } 	\sqrt{d(x,p_i)d(y,p_i)}\ge \sqrt{d(x)d(y)}.
$$
Now, with the help of \eqref{bernoulli} and the assumption, we deduce that
\begin{eqnarray*}
\hat{\tau}_D(x,y) &=& \frac{1}{k} \log \prod_{i=1}^k \left(1+\frac{2 d(x,y)}{\sqrt{d(x,p_i)d(y,p_i)}}\right) \le \frac{1}{k} \log \prod_{i=1}^k \left(1+ \frac{2 d(x,y)}{\sqrt{d(x)d(y)}}\right)\\
&=& \log \left(1+ \frac{2 d(x,y)}{\sqrt{d(x)d(y)}}\right) \le 2 \log \left(1+ \frac{ d(x,y)}{\sqrt{d(x)d(y)}}\right)\\
&\le & 2 \log \left(\frac{d(x)+ d(x,y)}{\sqrt{d(x)d(y)}}\right) = u_D(x,y). 
\end{eqnarray*}

We will prove the sharpness in the doubly punctured Euclidean space $D=\Rn\setminus\{-e_1,e_1\}$. Consider the points $x=-te_1, y=te_1, 0<t<1$.
Then,
$$ u_D(x,y)=2\log \left(\frac{1+t}{1-t}\right), \quad \hat{\tau}_D(x,y)=\log \left(1+\frac{4t}{\sqrt{1-t^2}}\right),
$$
and
$$ \lim_{t\to 0} \frac{u_D}{\zeta_D}=\lim_{t\to 0} \big(4t+\sqrt{1-t^2}\big)=1.
$$
The proof is complete.
\end{proof}

Next, we compare the scale-invariant Cassinian metric to the $\tilde{j}$-metric in once punctured metric spaces. 

\begin{theorem}\label{main3}
Let $(X,d)$ be a metric space, and let $D=X\setminus \{p\}, p\in X$. Then for $x,y\in D$,
$$\frac{1}{2}\tilde{j}_p(x,y)\le \tau_p(x,y)\le 2\tilde{j}_p(x,y).
$$
The inequalities are sharp for  $D=\Rn\setminus\{0\}$.
\end{theorem}

\begin{proof}
Without loss of generality assume that $d(x)\le d(y)$. Now,
\begin{eqnarray*}
\tau_p(x,y) &=& \log \left(1+2 \frac{d(x,y)}{\sqrt{d(x,p) d(y, p)}}\right)\le \log \left(1+2 \frac{d(x,y)}{d(x,p)}\right)\\
&\le & 2 \log \left(1+ \frac{d(x,y)}{d(x)}\right)= 2 \tilde{j}_p(x,y),
\end{eqnarray*}
where the first inequality follows from the assumption and the second inequality follows from \eqref{bernoulli}.

To establish the first inequality it is enough to prove that
$$ 1+\frac{d(x,y)}{d(x)}\le \left(1+2\frac{d(x,y)}{\sqrt{d(x)d(y)}}\right)^2.
$$ 
That is to show
$$ \frac{1}{d(x)}\le \frac{4d(x,y)}{d(x)d(y)}+\frac{4}{\sqrt{d(x)d(y)}}.
$$
Clearly,
\begin{eqnarray*}
\frac{4d(x,y)}{d(x)d(y)}+\frac{4}{\sqrt{d(x)d(y)}} &\ge & \frac{4d(x,y)}{d(x)d(y)}+\frac{4}{d(y)}= \frac{4}{d(y)} \left(1+\frac{d(x,y)}{d(x)}\right)\\
&\ge & \frac{4}{d(x)} \ge \frac{1}{d(x)}.
\end{eqnarray*}
Here the first inequality follows from the assumption and the second inequality follows from the triangle inequality.

For the sharpness, consider $D=\Rn\setminus\{0\}$ with $x=e_1$ and $y=te_1,0<t<1$. For this choice of points, we have
$$ \tilde{j}_D(x,y)=-\log t, \quad \mbox{ and } \tilde{\tau}_D(x,y)=\log \left(1+\frac{2(1-t)}{\sqrt{t}}\right).
$$
Clearly,
$$ \lim_{t\to 1} \frac{\tilde{j}_D}{\tilde{\tau}_D}=\lim_{t\to 1} \frac{\sqrt{t}+2(1-t)}{1+t}=\frac{1}{2},
$$
and 
$$ \lim_{t\to 0} \frac{\tilde{j}_D}{\tilde{\tau}_D}=\lim_{t\to 0} \frac{\sqrt{t}+2(1-t)}{1+t}=2.
$$
The proof is complete.
\end{proof}

The next result compares the average scale-invariant Cassinian metric with the $\tilde{j}$-metric in arbitrary metric spaces.
\begin{theorem}\label{main4}
Let $(X,d)$ be a metric space, and let $p_1,p_2,\ldots,p_k\in X$. Then, for $D=X\setminus \{p_1,p_2,\ldots,p_k\}$ and $x,y\in D$ we have
$$\hat{\tau}_D(x,y)\le 2\tilde{j}_D(x,y).
$$
The inequality is sharp for  $D=\Rn\setminus\{-e_1,e_1\}$.
\end{theorem}

\begin{proof}
Let $d(x,p):=\min\{d(x,p_i),d(y,p_i)\},i=1,2,\ldots,k$. Now,
\begin{eqnarray*}
\hat{\tau}_D(x,y) &=& \frac{1}{k} \sum_{i=1}^k \tau_{p_i}(x,y)=\frac{1}{k} \sum_{i=1}^k \log \left(1+2 \frac{d(x,y)}{\sqrt{d(x,p_i)d(y,p_i)}}\right)\\
 &=& \frac{1}{k} \log \prod_{i=1}^k \left(1+2 \frac{d(x,y)}{\sqrt{d(x,p_i)d(y,p_i)}}\right) \le \frac{1}{k} \log \prod_{i=1}^k \left(1+2 \frac{d(x,y)}{d(x,p)}\right)\\
 &=& \frac{1}{k} \log \left(1+2 \frac{d(x,y)}{d(x,p)}\right)^k = \log \left(1+2 \frac{d(x,y)}{d(x,p)}\right) \le 2 \log \left(1+ \frac{d(x,y)}{d(x,p)}\right)\\
 &=& 2 \tilde{j}_D(x,y).
\end{eqnarray*}
Here the first and the second inequalities respectively follow from the assumption and \eqref{bernoulli}. 

We will prove the sharpness part in the twice punctured Euclidean space $D=\Rn\setminus \{-e_1,e_1\}$. Choose the points $x=0$ and $y=te_1, 0<t<1$. Clearly,
$$ \tilde{j}_D(x,y)=-\log(1-t), \quad \zeta_D=\frac{1}{2} \left[\log\left(1+\frac{2t}{\sqrt{1-t}}\right)+
\log\left(1+\frac{2t}{\sqrt{1+t}}\right)\right],
$$
and
$$ \lim_{t\to 0} \frac{\tilde{j}_D}{\zeta_D}= \lim_{t\to 0}
-\frac{2}{{\left(t - 1\right)} {\left(\frac{\frac{2}{\sqrt{t + 1}} - \frac{t}{{\left(t + 1\right)}^{3/2}}}{\frac{2 \, t}{\sqrt{t + 1}} + 1} + \frac{\frac{t}{{\left(-t + 1\right)}^{3/2}} + \frac{2}{\sqrt{-t + 1}}}{\frac{2 \, t}{\sqrt{-t + 1}} + 1}\right)}}
=\frac{1}{2}.
$$
The proof is complete.
\end{proof}

We now compare the scale-invariant Cassinian metric with the $j$-metric in once punctured metric spaces.

\begin{theorem} \label{main5}
Let $(X,d)$ be a metric space and $D=X\setminus \{p\}, p\in X$. Then for $x,y\in D$
$$j_p(x,y)\le \tau_p(x,y)\le 2 j_p(x,y).
$$
The inequalities are  sharp for  $D=\Rn\setminus\{0\}$.
\end{theorem}

\begin{proof}
Without loss of generality assume that $d(x,p)\le d(y,p)$. To prove the first inequality it is enough to show that
$$\left(1+\frac{d(x,y)}{d(x,p)}\right) \left(1+\frac{d(x,y)}{d(y,p)}\right) \le \left(1+2\frac{d(x,y)}{\sqrt{d(x,p)d(y,p)}}\right)^2,
$$
i.e., 
$$ (d(x,p)+d(x,y))(d(y,p)+d(x,y))\le (\sqrt{d(x,p)d(y,p)}+2d(x,y))^2.
$$
To this end, it is sufficient to show that
$$ d(x,p)+d(y,p)\le 3 d(x,y)+ 4 \sqrt{d(x,p)d(y,p)}.
$$
Now,
\begin{eqnarray*}
d(x,p)+d(y,p)\le 2 d(x,p)+d(x,y) &\le& 2 \sqrt{d(x,p)d(y,p)}+ d(x,y)\\
&\le & 3 d(x,y)+ 4 \sqrt{d(x,p)d(y,p)},
\end{eqnarray*}
where the first inequality follows from the triangle inequality and second inequality follows from the assumption.

To establish the second inequality, we need to prove that
$$ \left(1+ 2 \frac{d(x,y)}{\sqrt{d(x,p)d(y,p)}} \right) \le \left(1+\frac{d(x,y)}{d(x,p)}\right) \left(1+\frac{d(x,y)}{d(y,p)}\right).
$$
i.e., 
$$ 2 \sqrt{d(x,p)d(y,p)} \le d(x,p)+d(y,p)+d(x,y)
$$
which follows from \eqref{am-gm}. 

To prove the sharpness, consider the punctured Euclidean space $\Rn\setminus\{0\}$ and choose the points $x=e_1$ and $y=te_1,\, 0<t<1$.
Simple computations yield,
$$ \tilde{\tau}_0(x,y)=\log \left(1+\frac{2(1-t)}{\sqrt{t}}\right), \quad 
j_0(x,y)=\frac{1}{2}\log \left(\frac{2-t}{t}\right).
$$
Therefore,
$$ \lim_{t\to 0} \frac{\tilde{\tau}_0}{j_0}= \lim_{t\to 0} \frac{(1+t)(2-t)}{\sqrt{t}+2(1-t)}=1
$$
and 
$$ \lim_{t\to 1} \frac{\tilde{\tau}_0}{j_0}= \lim_{t\to 1} \frac{(1+t)(2-t)}{\sqrt{t}+2(1-t)}=2.
$$
The proof is complete.
\end{proof}

The following result compares the average scale-invariant Cassinian metric with the $j$ metric in arbitrary metric spaces.
\begin{theorem} \label{main6}
Let $(X,d)$ be a metric space and $p_1,p_2,\ldots,p_k\in X$. Then, for $D=X\setminus \{p_1,p_2,\ldots,p_k\}$ and $x,y\in D$ we have
$$\hat{\tau}_D(x,y)\le 2 j_D(x,y).
$$
\end{theorem}

\begin{proof}
Suppose that $x,y\in D$. Then for all $i=1,2,\ldots,k$, $d(x)\le d(x,p_i)$ and $d(y)\le d(y,p_i)$. Here our aim is to show that $\tau_{p_i}(x,y)\le 2 j_D(x,y)$ for all $i=1,2,\ldots,k$. That is to show that 
$$ 2 d(x)d(y)\le (d(x)+d(y)+d(x,y))\sqrt{d(x,p_i)d(y,p_i)},
$$
which is true because
using \eqref{am-gm} and the assumption, it is easy to see that 
$$d(x)+d(y)+d(x,y)\ge d(x)+d(y)\ge 2 \sqrt{d(x)d(y)}$$ and $$\sqrt{d(x,p_i)d(y,p_i)}\ge \sqrt{d(x)d(y)}.
$$ 
Now,
$$ \hat{\tau}_D(x,y) = \frac{1}{k} \sum_{i=1}^k \tau_{p_i}(x,y)\le \frac{1}{k}
\sum_{i=1}^k 2 j_D(x,y)=2 j_D(x,y). 
$$  
\end{proof}

In addition to the $j_D$ and $\tilde{j}_D$ metrics, we study another metric which is generated from the $\tilde{j}_D$ metric and is defined for $x,y\in D$ by
$$
j^*_D(x,y):=\tanh \left(\frac{\tilde{j}_D(x,y)}{2}\right).
$$
The $j^*_D$ metric was studied in \cite{HVZ17}. We first compare the $j^*_D$ metric with the $\tau_p$ metric in once punctured spaces.

\begin{theorem} \label{main7}
Let $(X,d)$ be a metric space and $D=X\setminus \{p\},\, p\in X$. Then for $x,y\in D$
$$\frac{j^*_D(x,y)}{2}\le \tanh\left(\frac{\tau_p(x,y)}{2}\right)\le 2j^*_D(x,y).
$$
The right inequality is sharp for  $D=\Rn\setminus\{0\}$.
\end{theorem} 
\begin{proof}
Without loss of generality, assume that $d(x,p)\le d(y,p)$. It is easy to see that
\begin{equation}\label{eqj*}
d(x,p)\le \sqrt{d(x,p)d(y,p)} \le d(y,p).
\end{equation}
To prove $\tanh\left({\tau_p(x,y)}/{2}\right)\le 2j^*_D(x,y)$, it is enough to show that
$$ d(x,y)+2\sqrt{d(x,p) d(y,p)}\ge 2 d(x,p)
$$
which follows easily from \eqref{eqj*}.

Establishing the first inequality is equivalent to show
$$ d(x,y)+4 d(x,p)\ge \sqrt{d(x,p) d(y,p)}.
$$ 
Now, using the triangle inequality and \eqref{eqj*}, we obtain
$$ d(x,y)+4 d(x,p) \ge d(y,p)-d(x,p)+ 4 d(x,p) \ge \sqrt{d(x,p) d(y,p)}.
$$
For the sharpness, consider $D=\Rn\setminus\{0\}$ with $x=e_1$ and $y=te_1,0<t<1$. For this choice of points, we have
$$ j^*_D(x,y)=\frac{1-t}{1+t}, \quad \mbox{ and } \tanh \left(\frac{\tau_D(x,y)}{2}\right)=\frac{1-t}{1-t+\sqrt{t}}.
$$
Taking the limit $t\to 1$, we obtain that
$$ \lim_{t\to 1} \frac{\tanh(\tau_D(x,y)/2)}{j^*_D(x,y)}=\lim_{t\to 1} \frac{1+t}{1-t+\sqrt{t}}=2.
$$
The proof is complete. 
\end{proof}

The following theorem compares the $\tau_p$ metric with the triangular ratio metric
in once punctured spaces.
\begin{theorem} \label{main8}
Let $(X,d)$ be a metric space and $D=X\setminus \{p\},\, p\in X$. Then for $x,y\in D$
$$s_D(x,y)\le \frac{e^{\tau_p(x,y)}-1}{4}.
$$
The inequality is sharp for   $D=\Rn\setminus\{0\}$.
\end{theorem}
\begin{proof}
To establish the inequality it is enough to show that
$$ \frac{d(x,y)}{d(x,p)+d(y,p)} \le \frac{d(x,y)}{2 \sqrt{d(x,p) d(y,p)}},
$$
which is true by \eqref{am-gm}.

To prove the sharpness, consider the once punctured Euclidean space $D=\Rn\setminus\{0\}$. Let $x=e_1$ and $y=-e_1$. Then
$$ \tau_D(x,y)= \log 5 \mbox{ and } s_D(x,y)=1.
$$
The proof is complete.
\end{proof}
\section{The density function, lower and upper bounds}
In this section, we obtain the lower and upper bounds for the scale-invariant Cassinian metric in once punctured metric spaces. As a consequence, we obtain the density function for the average scale-invariant Cassinian metric in general metric spaces.

\begin{theorem}\label{lud}
Let $(X,d)$ be a metric space and $D=X\setminus \{p\},p\in X$. Suppose that $x\in D$. Then
$$ \log \left(1+\frac{2d(x,y)}{d(x,p)+d(x,y)} \right) \le \tau_p(x,y) \le 
\log \left(1+\frac{2d(x,y)}{d(x,p)-d(x,y)} \right)
$$
for $y\in B(x,d(x,p))$. Moreover,
$$ \lim_{y\to x} \frac{\tau_p(x,y)}{d(x,y)}=\frac{2}{d(x,p)}.
$$
\end{theorem}

\begin{proof}
Suppose that $d(x,y)\le d(x,p)$. Now,
\begin{eqnarray*} 
\tau_p(x,y) =\log \left(1+\frac{2d(x,y)}{\sqrt{d(x,p)d(y,p)}} \right) &\le & \log \left(1+ \frac{2d(x,y)}{\sqrt{d(x,p) (d(x,p)-d(x,y))}} \right)\\
&\le & \log \left(1+ \frac{2d(x,y)}{d(x,p)-d(x,y)} \right).
\end{eqnarray*}
On the other hand,
$$ \sqrt{d(x,p)d(y,p)}\le \frac{d(x,p)+d(y,p)}{2}\le \frac{2d(x,p)+d(x,y)}{2} \le d(x,p)+d(x,y).
$$
Hence,
$$ \tau_p(x,y) =\log \left(1+\frac{2d(x,y)}{\sqrt{d(x,p)d(y,p)}} \right) \ge 
\log \left(1+ \frac{2d(x,y)}{d(x,p)+d(x,y)} \right).
$$
The second conclusion follows from the fact that 
\begin{eqnarray*}
\lim_{y\to x} \frac{\log \left(1+\ds\frac{2d(x,y)}{d(x,p)+d(x,y)}\right)}{d(x,y)}  
=\frac{2}{d(x,p)}= \lim_{y\to x} \frac{\log \left(1+\ds\frac{2d(x,y)}{d(x,p)-d(x,y)} \right)}{d(x,y)}.
\end{eqnarray*}
The proof is complete.
\end{proof}

The following result yields some lower and upper bounds for the average scale-invariant Cassinian metric in arbitrary metric spaces.
\begin{theorem}\label{den-gen}
Let $(X,d)$ be a metric space and $p_1,p_2,\ldots,p_k\in X$. Then, for $D=X\setminus \{p_1,p_2,\ldots,p_k\}$ and $x\in D$ we have
$$ 
\log \left(1+\frac{2d(x,y)}{d(x)+d(x,y)} \right) \le \hat{\tau}_D(x,y) \le 
\log \left(1+\frac{2d(x,y)}{d(x)-d(x,y)} \right)
$$
for $y\in B(x,d(x))$.
Moreover, 
$$ \lim_{y\to x} \frac{\hat{\tau}_D(x,y)}{d(x,y)}= \frac{2}{d(x)}.
$$
\end{theorem}

\begin{proof}
By the triangle inequality, it follows that 
$$ d(y,p_i)\ge d(x,p_i)-d(x,y) \ge d(x)-d(x,y) \mbox{ for all } i=1,2,\ldots,k.
$$
Now, $\sqrt{d(x,p_i)d(y,p_i)}\ge \sqrt{d(x) (d(x)-d(x,y))}$. Since $d(x,y)\le d(x)$, it follows that
\begin{eqnarray*}
\zeta_D (x,y) =\frac{1}{k} \sum_{i=1}^k \log \left(1+ \frac{2d(x,y)}{\sqrt{d(x,p_i)d(y,p_i)}} \right) &\le& \frac{1}{k} \sum_{i=1}^k  \log \left(1+\frac{2d(x,y)}{\sqrt{d(x) (d(x)-d(x,y))}} \right)\\
&\le & \log \left(1+ \frac{2d(x,y)}{d(x)-d(x,y)}\right).
\end{eqnarray*}

\end{proof}
\section{Inclusion properties}
In this section we are interested in the metric ball inclusion relations associated with the scale-invariant Cassinian metric.  First, we prove the metric ball inclusions in once punctured spaces.

\begin{theorem} \label{main8b}
Let $(X,d)$ be a metric space and $D=X\setminus \{p\},\,p\in X$. Then the following inclusion property holds true for $t\in [0,\log 3)$:
$$ B_d(x,r)\subseteq B_{\tau_p}(x,t)\subseteq B_d(x,R),
$$
where $r=\ds\frac{e^t-1}{e^t+1}d(x)$ and $R=\ds\frac{e^t-1}{3-e^t} d(x)$.

Moreover, the limit of $R/r$ tends to $1$ when $t\to 0$. 
\end{theorem}
\begin{proof}
Suppose that $x\in D$ and $y\in B_d(x,d(x))$. We will make use of Theorem~\ref{lud} to establish the relation. Let $y\in B_{\tau_p}(x,t)$. Then, by Theorem~\ref{lud},
$$ \log \left(1+\frac{2d(x,y)}{d(x,p)-d(x,y)}\right)<t, \mbox{ i.e. } d(x,y)< \frac{e^t-1}{3-e^t} d(x).
$$
That is, $y\in B_d(x,R)$ with $R= ((e^t-1)/(3-e^t)) d(x)$. This proves the second inclusion relation. Now, suppose that $y\in B_d(x,r)$ with $r=((e^t-1)/(e^t+1))d(x)$. Then from Theorem~\ref{lud}, we have $\tau_p(x,y)<t$. 
Hence, $B_d(x,r)\subseteq B_{\tau_p}(x,t)$.

Clearly,
$$ \lim_{t\to 0} \frac{R}{r}=\lim_{t\to 0} \frac{e^t+1}{3-e^t}=1.
$$ 
The proof is complete.
\end{proof}

\begin{theorem}
	Let $(X,d)$ be a metric space. For $k\ge 1$, let $p_1,p_2,\cdots,p_k \in X$ and $D=X\setminus \{p_1,p_2,\cdots,p_k\}$. Then  for $t\in [0,\log 3)$ the following inclusion property holds true:
	$$ B_d(x,r)\subseteq B_{\zeta_D}(x,t)\subseteq B_d(x,R),
	$$
	where $r=\ds\frac{e^t-1}{e^t+1}d(x)$, $R=\ds\frac{e^t-1}{3-e^t} d(x)$.
	
	Moreover, the limit of $R/r$ tends to $1$ when $t\to 0$. 
\end{theorem}
\begin{proof}
The proof follows from Theorem~\ref{den-gen}.
\end{proof}

\subsection{Lipschitz properties}
In this section we also consider the behavior of the $\tau_p$ metric under bilipschitz mappings. Recall that a mapping $f:(X,d_1)\to (Y,d_2)$ is $L$-bilipschitz if there exists a real number $L\ge 1$ such that
 for any $x,y\in X$ 
\begin{equation}\label{bilip}
\frac{d_1(x,y)}{L}\le d_2(f(x),f(y))\le L d_1(x,y).
\end{equation}
 
The following lemma shows that $L$-bilipschitz mappings with respect to the induced metric is $L^2$ bilipschitz with respect to the $\tau_p$ metric.

\begin{lemma}
Let $(X,d_1)$ and $(Y,d_2)$ be metric spaces. Suppose that $f:(X,d_1)\to (Y,d_2)$ is $L$-bilipschitz mapping. Then
$$ L^{-2} \tau_p(x,y)\le \tau_{f(p)} (f(x),f(y))\le L^2 \tau_p (x,y)
$$ 
for $x,y,p\in X$.
\end{lemma}

\begin{proof}
We will prove the right inequality. Suppose that $x,y\in X\setminus\{p\}$. Then using the definition of a bilipschitz mapping \eqref{bilip} and the Bernoulli inequality \eqref{bernoulli} we see that
\begin{eqnarray*}
\tau_{f(p)}(f(x),f(y)) &=& \log \left(1+ 2 \frac{d_2(f(x),f(y))}{\sqrt{d_2(f(x),f(p)) d_2(f(y),f(p))}} \right)\\
&\le & \log \left(1+ 2L^2 \frac{d_1(x,y)}{\sqrt{d_1(x,p) d_1(y,p)}} \right)\\
&\le & L^2 \log \left(1+ 2 \frac{d_1(x,y)}{\sqrt{d_1(x,p) d_1(y,p)}} \right)=L^2 \tau_p(x,y).
\end{eqnarray*}
The left inequality follows immediately since the inverse of a bilipschitz mapping is bilipschitz. The proof is complete.
\end{proof}

\begin{theorem} \label{main9}
Let $(X,d_1)$ and $(Y,d_2)$ be metric spaces. Suppose that $f:(X,d_1)\to (Y,d_2)$ is a homeomorphism and there exists $L\ge 1$ such that for all $x,y\in X$
$$ \tau_p(x,y)/L \le \tau_{f(p)} (f(x),f(y))\le L \tau_p(x,y),
$$
where $p\in X$. Then $f$ is quasiconformal with linear dilatation $H(f)\le L^2$. 
\end{theorem}

\begin{proof}
Let $p\in X$. Fix $z\in X\setminus\{p\}$, $t\in (0,1/2)$ and $x,y\in X\setminus\{p\}$ with
$$ d_1(x,z)=d_1(y,z)=td(z,p).
$$
Observe that 
$$ \frac{d_2(f(x),f(y))\cdot(d_2(f(x),f(p))\wedge d_2(f(y),f(p)))}{\sqrt{d_2(f(x),f(p)) d_2(f(y),f(p))}}  \le d_2(f(x),f(y))
$$
$$ \le \frac{d_2(f(x),f(y))\cdot(d_2(f(x),f(p))\wedge d_2(f(y),f(p)))}{\sqrt{d_2(f(x),f(p)) d_2(f(y),f(p))}}  \sqrt{1+\frac{d_2(f(x),f(y))}{d_2(f(x),f(p))\wedge d_2(f(y),f(p))}},
$$
where the second inequality follows from the triangle inequality.
For this choice of points $x$ and $y$, it is easy to see that
$$\tau_p(x,y)\le \log \left(\frac{1+3t}{1-t}\right),\\
\tau_p(x,y)\ge \log \left(\frac{1-3t}{1+t}\right). 
$$
Now, 
\begin{eqnarray*}
\frac{d_2(f(x),f(z))}{d_2(f(y),f(z))} &\le & 
\ds\frac{\frac{d_2(f(x),f(z)).(d_2(f(x),f(p))\wedge d_2(f(z),f(p)))}{\sqrt{d_2(f(x),f(p)) d_2(f(z),f(p))}}  \sqrt{1+\frac{d_2(f(x),f(z))}{d_2(f(x),f(p))\wedge d_2(f(z),f(p))}}}{\frac{d_2(f(y),f(z))\cdot(d_2(f(y),f(p))\wedge d_2(f(z),f(p)))}{\sqrt{d_2(f(y),f(p)) d_2(f(z),f(p))}}}\\
&\le & \frac{e^{\tau_{f(p)}(f(x),f(z))}-1}{e^{\tau_{f(p)}(f(y),f(z))}-1} \\ 
&& \cdot \frac{(d_2(f(x),f(p))\wedge d_2(f(z),f(p)))\cdot \sqrt{1+\frac{d_2(f(x),f(z))}{d_2(f(x),f(p))\wedge d_2(f(z),f(p))}}}{(d_2(f(y),f(p))\wedge d_2(f(z),f(p)))}\\
&\le & \frac{e^{L \tau_{p}(x,z)}-1}{e^{\tau_{p}(y,z)/L}-1}  \\
&& \cdot\frac{(d_2(f(x),f(p))\wedge d_2(f(z),f(p)))\cdot \sqrt{1+\frac{d_2(f(x),f(z))}{d_2(f(x),f(p))\wedge d_2(f(z),f(p))}}}{(d_2(f(y),f(p))\wedge d_2(f(z),f(p)))}\\
&\to & L^2,
\end{eqnarray*}
when $t\to 0$. The proof is complete.
\end{proof}
\section{Convexity properties of metric balls}
In this section we establish the convexity properties of the $\tau_p$ metric balls in Euclidean punctured spaces. 
For the proof we follow a technique similar to the one used in \cite{HKLV16}.

Suppose that $r=r(\theta)$ is a function given in polar coordinates. Then
the slope of the curve $r=r(\theta)$ at the point $(r, \theta)$ is
$$
m(\theta)=\frac{r(\theta)+\tan (\theta) \ds\frac{d r}{d \theta}}{-r(\theta) \tan (\theta)+\ds\frac{d r}{d \theta}}.
$$

\begin{theorem}\label{non-conv}
Let $D=\Rn\setminus\{z\},z\in \Rn, x\in D$ and $r\in (0,1]$. Then the metric ball $B_{\tau_p}(x,r)$ is not convex for all $r> \log 3$. 
\end{theorem}

\begin{proof}
By symmetric property, without loss of generality we consider the balls in the punctured space $\R^2\setminus\{0\}$ with $x=1$. Let $(0,1]$ and $y\in \partial B_{\tau_p}(1,r)$. Then by definition we have
$$ \tau_0(1,y)=\log \left(1+2\frac{|1-y|}{\sqrt{|y|}}\right)=r
$$
implies
$$ |1-y|= \frac{e^r-1}{2} \sqrt{|y|}.
$$
Again from the law of cosines we have $|1-y|^2=|y|^2+1-2|y|\cos \theta$, where $\theta$ is the angle between $y,0$ and $x$. Hence, by denoting $t=|y|$, we get
\begin{equation}\label{non-conv-tp}
t^2+1-2t\cos \theta= \frac{(e^r-1)^2}{4}t.
\end{equation}
Solving \eqref{non-conv-tp} for $t$ we obtain
\begin{equation}\label{sol-ball-conv}
t(\theta)=\frac{1}{2}\left[2\cos \theta + \frac{(e^r-1)^2}{4} \pm \sqrt{\left(2\cos \theta+\frac{(e^r-1)^2}{4}\right)^2-4}\right].
\end{equation}
For the proof we consider the following value of $t(\theta)$
$$ t_1(\theta)=\frac{1}{2}\left[2\cos \theta + \frac{(e^r-1)^2}{4} - \sqrt{\left(2\cos \theta+\frac{(e^r-1)^2}{4}\right)^2-4}\right].
$$
We are interested in those values of $\theta$ for which
$$ \left(2\cos \theta+\frac{(e^r-1)^2}{4}\right)^2-4>0
$$ 
or,
$$ \cos \theta > \left(1-\frac{(e^r-1)^2}{8}\right), \mbox{ or }, \theta <
\arccos \left(1-\frac{(e^r-1)^2}{8}\right).
$$
Clearly, for $r>\log 3$, $\arccos \left(1-(e^r-1)^2/8\right)>\pi/3$. 
Therefore, it is enough to consider the angles $\theta \in (0,\pi/3)$.
Consider the slope of the tangent to the boundary $\partial B_{\tau_p}(1,r)$ of the metric ball near $\theta=0$. The formula for the slope of the tangent to the boundary is
\begin{eqnarray*}
m_1(\theta) &=& \frac{t_1(\theta)+\tan \theta t_1'(\theta)}{-t_1(\theta)\tan \theta+t_1'(\theta)}\\
&=& \frac{2\sin \theta \tan \theta +\sqrt{\left(2\cos \theta+\frac{(e^r-1)^2}{4}\right)^2-4}}{2\sin \theta -\tan \theta \sqrt{\left(2\cos \theta+\frac{(e^r-1)^2}{4}\right)^2-4}},
\end{eqnarray*}
 where 
$$ t_1'(\theta)=-\sin \theta \left[1-\frac{2\cos \theta +\frac{(e^r-1)^2}{4}}{\sqrt{\left(2\cos \theta +\frac{(e^r-1)^2}{4}\right)^2-4}}\right].
$$
Clearly, when $\theta\in (0,\pi/3)$ and $r> \log 3$,
$$ 2\sin \theta \tan \theta +\sqrt{\left(2\cos \theta+\frac{(e^r-1)^2}{4}\right)^2-4}>0.
$$
We only need to show that 
$$ g(\theta):=2\sin \theta -\tan \theta \sqrt{\left(2\cos \theta+\frac{(e^r-1)^2}{4}\right)^2-4}<0
$$
near $\theta=0$. Simple computations show that $g(0)=0$ and 
$$g'(0)=2+\sqrt{\left(2+\frac{(e^r-1)^2}{4}\right)^2-4}>0
$$ if
$r> \log 3$. Hence for sufficiently small $\theta$, $g(\theta)>0$.
This shows that the slope of the tangent is negative near $\theta=0$. The proof is complete.
\end{proof}

The following lemma is useful in this setting.
\begin{lemma}\label{lem-conv}
For $0<r\le \log 3$ and $\theta\in (0,\arccos(1-(e^r-1)^2/8))$,
$$4 \sin^2 \theta \sqrt{\alpha(\theta)}+\sin \theta\alpha'(\theta)-2\cos \theta \alpha(\theta)+ (\alpha(\theta))^{3/2}\le 0,
$$
where $\alpha(\theta)=\left(2\cos \theta+\frac{(e^r-1)^2}{4}\right)^2-4$.
\end{lemma}
\begin{theorem} \label{main10}
The $\tau_p$-metric ball $B_{\tau_p}(x,r)$ is convex for $r\le \log 3$.  
\end{theorem}
\begin{proof}
We consider the case when $n=2$. Following \eqref{sol-ball-conv} we denote
$$ t_1(\theta)=\frac{1}{2}\left[2\cos \theta + \frac{(e^r-1)^2}{4} - \sqrt{\left(2\cos \theta+\frac{(e^r-1)^2}{4}\right)^2-4}\right],
$$
and
$$ t_2(\theta)=\frac{1}{2}\left[2\cos \theta + \frac{(e^r-1)^2}{4} + \sqrt{\left(2\cos \theta+\frac{(e^r-1)^2}{4}\right)^2-4}\right].
$$

As it is shown in Theorem~\ref{non-conv} the slope of the tangent to the boundary of the ball $\partial B_{\tau_p}(x,r)$ is
$$m_1(\theta) = \frac{2\sin \theta \tan \theta +\sqrt{\left(2\cos \theta+\frac{(e^r-1)^2}{4}\right)^2-4}}{2\sin \theta -\tan \theta \sqrt{\left(2\cos \theta+\frac{(e^r-1)^2}{4}\right)^2-4}},
$$
and
$$m_2(\theta) = \frac{2\sin \theta \tan \theta -\sqrt{\left(2\cos \theta+\frac{(e^r-1)^2}{4}\right)^2-4}}{2\sin \theta +\tan \theta \sqrt{\left(2\cos \theta+\frac{(e^r-1)^2}{4}\right)^2-4}}.
$$
It follows that
$$ m_1'(\theta)=\frac{\sec^2\Big[ \theta \Big(4 \sin^2\big(\theta \sqrt{\alpha(\theta)}\big)+\sin\big( \theta\alpha'(\theta)\big)-2\cos \big(\theta \alpha(\theta)\big)+ \big(\alpha(\theta)\big)^{3/2}\Big)\Big]}
{\sqrt{\alpha(\theta)}\big[2\sin \theta-\tan\big( \theta \sqrt{\alpha(\theta)}\big)\big]^2},
$$
where 
$$ \alpha(\theta)=\left(2\cos \theta+\frac{(e^r-1)^2}{4}\right)^2-4.
$$

Our aim is to show that for all $r\in (0,\log 3]$ and for all 
$$
0<\theta < \arccos(1-(e^r-1)^2/8), 
$$
we have $m_1'(\theta)\le 0$ and $m_2'(\theta)\ge 0$. Note that $ \alpha(\theta)>0$ if and only if
\begin{equation}
\label{eq-conv}
 \theta\in \big(0, \arccos(1-(e^r-1)^2/8)\big).
\end{equation}

It is clear that for the specified range of $r$ and $\theta$, 
$$\sqrt{\alpha(\theta)}(\sin \theta-\tan \theta \sqrt{\alpha(\theta)})^2>0.
$$
Lemma~\ref{lem-conv} guarantees that
$$2\sec^2\Big[ \theta \Big(4 \sin^2\big( \theta \sqrt{\alpha(\theta)}\big)+\sin \big( \theta\alpha'(\theta))\big)-2\cos\big( \theta \alpha(\theta)\big)+ \big(\alpha(\theta)\big)^{3/2}\Big)\Big]\le 0
$$
Now,
$$
m_2'(\theta)=\frac{\sec^2 \Big[\theta \Big(4 \sin^2\big( \theta \sqrt{\alpha(\theta)}\big)-\sin\big( \theta\alpha'(\theta)\big)+2\cos\big( \theta \alpha(\theta)\big)+ \big(\alpha(\theta)\big)^{3/2}\Big)\Big]}
{\sqrt{\alpha(\theta)}\big[2\sin \theta+\tan\big( \theta \sqrt{\alpha(\theta)}\big)\big]^2}.
$$
Moreover,  for the specified range $r\in (0,\log 3]$ and $\theta\in (0,\arccos(1-(e^r-1)^2/8))$ we have
$$ 4 \sin^2 \theta \sqrt{\alpha(\theta)}-\sin \theta\alpha'(\theta)+2\cos \theta \alpha(\theta)+ (\alpha(\theta))^{3/2}>0,
$$
because of \eqref{eq-conv} and the fact that 
$$ \alpha'(\theta)=-4\sin \theta \left(2\cos \theta+\frac{(e^r-1)^2}{4}\right).
$$
Hence $m_2'(\theta)\ge 0$. 
\begin{figure}[h!]
\includegraphics{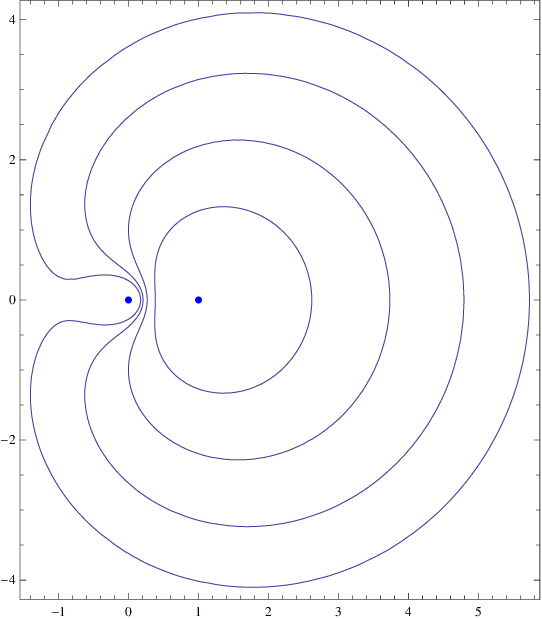}
\caption{The $\tau_p$-metric balls in $\R^2\setminus \{0\}$ centred at $e_1=(1,0)$ with radii $r=\log 3, \log 5,\log 7$, and $\log 8.8$.}
\end{figure}
The proof is complete.
\end{proof}

\subsection*{Acknowledgements}  The research was partially supported by the Natural Science Foundation of Guangdong Province (Grant no. 2024A1515010467) and the Li~Ka~Shing Foundation GTIIT-STU Joint Research Grant (Grant no. 2024LKSFG06).

\end{document}